\pdfoutput=1
\documentclass[11pt,a4paper,reqno]{amsart}
\usepackage[british]{babel}

\usepackage{kpfonts}

\usepackage[T1]{fontenc}

\usepackage{setspace}
\usepackage{relsize}
\usepackage{amssymb}
\usepackage{amsmath}
\usepackage{amsthm}
\usepackage{enumitem}

\usepackage[colorlinks,citecolor=blue,urlcolor=black,linkcolor=blue,linktocpage]{hyperref}
\usepackage{mathtools}

\usepackage[pic,curve,cmtip]{xy}

\makeatletter
\@namedef{subjclassname@2020}{\textup{2020} Mathematics Subject Classification}
\makeatother

\newtheorem{thm}{Theorem}[section]
\newtheorem*{mthm}{Main Theorem}
\newtheorem{cor}[thm]{Corollary}
\newtheorem{lem}[thm]{Lemma}

\newtheorem{prop}[thm]{Proposition}

\theoremstyle{definition}

\newtheorem{remark}[thm]{Remark}

\renewcommand{\epsilon}{\varepsilon}
\renewcommand{\phi}{\varphi}
\newcommand{\defeq}{\mathrel{\mathop:}=}

\renewcommand{\Re}{\operatorname{Re}}

\DeclareMathOperator{\Mat}{M}
\DeclareMathOperator{\lat}{L}
\DeclareMathOperator{\aff}{R}
\DeclareMathOperator{\vN}{N}
\DeclareMathOperator{\Pro}{P}
\DeclareMathOperator{\B}{B}
\DeclareMathOperator{\U}{U}
\DeclareMathOperator{\GL}{GL}
\DeclareMathOperator{\UCB}{UCB}
\DeclareMathOperator{\N}{\mathbb{N}}
\DeclareMathOperator{\Z}{\mathbb{Z}}
\DeclareMathOperator{\Q}{\mathbb{Q}}
\DeclareMathOperator{\R}{\mathbb{R}}
\DeclareMathOperator{\C}{\mathbb{C}}

\DeclareMathOperator{\Prob}{Prob}
\DeclareMathOperator{\Sph}{\mathbb{S}}
\DeclareMathOperator{\rk}{rk}
\DeclareMathOperator{\rank}{rank}
\DeclareMathOperator{\tr}{tr}
\DeclareMathOperator{\Sym}{Sym}
\DeclareMathOperator{\Cent}{C}
\DeclareMathOperator{\free}{F}

\begin{document}

\setlist{noitemsep}

\author{Friedrich Martin Schneider}
\address{F.M.S., Institute of Discrete Mathematics and Algebra, TU Bergakademie Freiberg, 09596 Freiberg, Germany}
\email{martin.schneider@math.tu-freiberg.de}

\title[Inner amenability and continuous rings]{Group von Neumann algebras, inner amenability, \\ and unit groups of continuous rings}
\date{\today}

\begin{abstract} 
	We prove that, if a discrete group $G$ is not inner amenable, then the unit group of the ring of operators affiliated with the group von Neumann algebra of $G$ is non-amenable with respect to the topology generated by its rank metric. This provides examples of non-discrete irreducible, continuous rings (in von Neumann's sense) whose unit groups are non-amenable with regard to the rank topology. Our argument establishes and uses connections with Eymard--Greenleaf amenability of the action of the unitary group of a $\mathrm{II}_{1}$ factor on the associated space of projections of a fixed trace.
\end{abstract}



\maketitle


\section{Introduction}

In a seminal work~\cite{VonNeumannBook}, von Neumann discovered a continuous analogue of finite-dimensional projective geometry. \emph{Continuous geometries}, i.e., complete, complemented, modular lattices whose algebraic operations possess certain natural continuity properties, are the central objects of this theory. A cornerstone in von Neumann's study is his \emph{coordinatization theorem}~\cite{VonNeumannBook}, which states that, firstly, the set $\lat(R)$ of all principal right ideals of every \emph{regular} ring $R$, ordered by set-theoretic inclusion, constitutes a complemented, modular lattice, and secondly, every complemented, modular lattice of an order at least four arises in this way from an up to isomorphism unique regular ring. A \emph{continuous} ring is a regular ring $R$ whose corresponding lattice $\lat(R)$ is a continuous geometry. Building on a dimension theory for (directly) irreducible continuous geometries, another profound achievement of~\cite{VonNeumannBook}, von Neumann proved that an irreducible, regular ring $R$ is continuous if and only if there exists a (necessarily unique) rank function $\rk \colon R \to [0,1]$, and that in such case $R$ is complete with respect to the induced \emph{rank metric} $R \times R \to [0,1], \, (a,b) \mapsto \rk(a-b)$. Thus, any irreducible, continuous ring $R$ admits a natural topology---the \emph{rank topology} generated by its rank metric---which turns $R$ into a topological ring.

While the discrete irreducible, continuous rings are precisely the ones isomorphic to a matrix ring $\Mat_{n}(D)$ for some division ring $D$ and some positive integer $n$ (see Remark~\ref{remark:matrix.rings}), the class of \emph{non-discrete} irreducible, continuous rings appears intriguingly vast. The initial example of an irreducible continuous geometry is the projection lattice of an arbitrary von Neumann factor $M$ of type $\mathrm{II}_{1}$, in which case the corresponding irreducible, continuous ring is non-discrete and can be described as the algebra $\aff(M)$ of densely defined, closed, linear operators \emph{affiliated with $M$}~\cite{MurrayVonNeumann}. For another example, given a division ring $D$, one may consider the inductive limit $\varinjlim\limits \Mat_{2^{n}}(D)$ of matrix rings \begin{displaymath}
	D \, \cong \, \Mat_{2^{0}}(D) \, \stackrel{\iota_{0}}{\longrightarrow} \, \ldots \, \stackrel{\iota_{n-1}}{\longrightarrow} \, \Mat_{2^{n}}(D) \, \stackrel{\iota_{n}}{\longrightarrow} \, \Mat_{2^{n+1}}(D) \, \stackrel{\iota_{n+1}}{\longrightarrow} \, \ldots
\end{displaymath} along the embeddings \begin{displaymath}
	\iota_{n} \colon \, \Mat_{2^{n}}(D) \, \longrightarrow \, \Mat_{2^{n+1}}(D), \quad a \, \longmapsto \, \begin{pmatrix} a & 0 \\ 0 & a \end{pmatrix} \qquad (n \in \N) .
\end{displaymath} Since the maps $(\iota_{n})_{n \in \N}$ are isometric with respect to the normalized rank\footnote{See~\cite[I.10.12, p.~359--360]{bourbaki} for details concerning the rank of matrices over division rings.} metrics \begin{displaymath}
	d_{n} \colon \, \Mat_{2^{n}}(D) \times \Mat_{2^{n}}(D) \, \longrightarrow \, [0,1], \quad (a,b) \, \longmapsto \, \tfrac{\rank(a-b)}{2^{n}} \qquad (n \in \N) ,
\end{displaymath} those metrics admit a joint extension to $\varinjlim\limits \Mat_{2^{n}}(D)$. The completion $\Mat_{\infty}(D)$ of $\varinjlim\limits \Mat_{2^{n}}(D)$ with respect to the resulting metric constitutes a non-discrete irreducible, continuous ring~\cite{NeumannExamples,HalperinInductive}. An abstract characterization of continuous rings arising in this manner can be found in~\cite{AraClaramunt18}.

There has been recent interest in concrete occurrences of continuous rings, for instance, in the context of with Kaplansky's direct finiteness conjecture~\cite{ElekSzabo,linnell} and the Atiyah conjecture~\cite{LinnellSchick,elek}. The present note is concerned with topological dynamics of the unit group $\GL(R)$ of an irreducible, continuous ring $R$, equipped with the relative rank topology. In~\cite{CarderiThom}, Carderi and Thom showed that, if $F$ is a finite field, then the topological group $\GL(\Mat_{\infty}(F))$ is \emph{extremely amenable}, i.e., every continuous action of $\GL(\Mat_{\infty}(F))$ on a non-void compact Hausdorff space has a fixed point. By work of the present author~\cite[Cor.~1.6]{FMS22}, for every non-discrete irreducible, continuous ring $R$, the union of extremely amenable topological subgroups of $\GL(R)$ is dense in $\GL(R)$. This illustrates that the phenomenon of extreme amenability is---to some extent---inherent to topological unit groups of non-discrete irreducible, continuous rings. On the other hand, by a well-known consequence of the ping-pong lemma, for every division ring $D$ of characteristic zero and every natural number $n \geq 2$, the unit group of the discrete irreducible, continuous ring $\Mat_{n}(D)$ is non-amenable, which raises the question as to whether there exist non-discrete irreducible, continuous rings with topologically non-amenable unit groups, too. This question is answered affirmatively by our main result, which concerns the ring of densely defined, closed, linear operators affiliated with the group von Neumann algebra $\vN(G)$ of a discrete group $G$.

\begin{mthm}[Corollary~\ref{corollary:main}] Let $G$ be a group that is not inner amenable.\footnote{Examples of such groups are given in Proposition~\ref{proposition:robin} and Theorem~\ref{theorem:haagerup.olesen}.} Then $\aff(\vN(G))$ is a non-discrete irreducible, continuous ring whose unit group is non-amenable with respect to the rank topology. \end{mthm}

The argument proving our main result proceeds via inspecting several isometric group actions for Eymard--Greenleaf amenability. More precisely, if $G$ is a non-inner amenable group and $t \in (0,1)$, then the natural action of $G$ on the space of projections of trace $t$ of the $\mathrm{II}_{1}$ factor $\vN(G)$, equipped with the trace metric, is not Eymard--Greenleaf amenable (Theorem~\ref{theorem:first}), which witnesses non-amenability of the topological group $\GL(\aff(\vN(G)))$, by virtue of a general mechanism comparing certain actions of $\GL(\aff(M))$ and the unitary group $\U(M) \leq \GL(\aff(M))$ for an arbitrary $\mathrm{II}_{1}$ factor $M$ (Lemma~\ref{lemma:amenable}).

This article is organized as follows. After recollecting some general background material on topological dynamics in Section~\ref{section:amenability}, we turn to continuous geometries and unit groups of their coordinate rings in Section~\ref{section:continuous.rings}. The subsequent Section~\ref{section:projections} contains a discussion of Eymard--Greenleaf amenability for actions of unitary groups of $\mathrm{II}_{1}$ factors on the associated projection spaces. In Section~\ref{section:groups}, we specify to group von Neumann algebras and connect our previous considerations with inner amenability of discrete groups, finishing the proof of our main result.

\section{Eymard--Greenleaf amenability}\label{section:amenability}

An action $G \curvearrowright (X,\mathscr{E})$ of a group $G$ by isomorphisms on a uniform space $(X,\mathscr{E})$ is said to be \emph{Eymard--Greenleaf amenable}\footnote{This term was coined by Pestov~\cite[Def.~3.5.9, p.~64]{PestovBook}, referencing works of Eymard~\cite{eymard} and Greenleaf~\cite{greenleaf}.} if the algebra \begin{displaymath}
	\UCB(X,\mathscr{E}) \, \defeq \, \{ f \in \ell^{\infty}(X,\R) \mid \forall \epsilon \in \R_{>0} \, \exists E \in \mathscr{E} \, \forall (x,y) \in E \colon \, \vert f(x) - f(y) \vert \leq \epsilon \}
\end{displaymath} of all uniformly continuous bounded real-valued functions on $(X,\mathscr{E})$ admits a $G$-invariant mean, i.e., a positive unital linear map \begin{displaymath}
\mu \colon \, \UCB(X,\mathscr{E}) \, \longrightarrow \, \R
\end{displaymath} such that \begin{displaymath}
	\forall g \in G \ \forall f \in \UCB(X,\mathscr{E}) \colon \quad \mu(f \circ \tilde{g}) \, = \, \mu(f) ,
\end{displaymath} where we let $\tilde{g} \colon X \to X, \, x \mapsto gx$ for each $g \in G$. In particular, this yields a concept of amenability for isometric group actions on metric spaces, where a metric space $(X,d)$ is being viewed as a uniform space carrying the induced uniformity \begin{displaymath}
	\{ E \subseteq X \times X \mid \exists r \in \R_{>0} \, \forall x,y \in X \colon \, d(x,y) < r \Longrightarrow (x,y) \in E \} .
\end{displaymath} Furthermore, Eymard--Greenleaf amenability naturally gives rise to a notion of amenability for topological groups. To be more precise, let $G$ be a topological group. Considering the neighborhood filter $\mathscr{U}(G)$ of the neutral element in $G$, one may endow $G$ with its \emph{right uniformity} \begin{displaymath}
	\mathscr{E}_{\Rsh}(G) \, \defeq \, \left\{ E \subseteq G \times G \left\vert \, \exists U \in \mathscr{U}(G) \, \forall x,y \in G \colon \, xy^{-1}\! \in U \Longrightarrow \, (x,y) \in E \right\} \right.\! .
\end{displaymath} The topological group $G$ is called \emph{amenable} if the action of the group $G$ by left translations on the uniform space $(G,\mathscr{E}_{\Rsh}(G))$ is Eymard--Greenleaf amenable. By a result of Rickert~\cite[Thm.~4.2]{rickert}, the topological group $G$ is amenable if and only if every continuous\footnote{Continuity of an action means \emph{joint} continuity.} action of $G$ on a non-void compact Hausdorff space admits an invariant regular Borel probability measure, or equivalently, if every continuous action of $G$ by affine homeomorphisms on a non-void compact convex subset of a locally convex topological vector space has a fixed point.

A topological group is said to have \emph{small invariant neighborhoods} if its neutral element admits a neighborhood basis consisting of conjugation-invariant subsets. The following is well known.

\begin{lem}\label{lemma:pestov} Let $G$ be an amenable topological group having small invariant neighborhoods. Then every continuous isometric action of $G$ on a non-empty metric space is Eymard--Greenleaf amenable. \end{lem}

\begin{proof} Consider a continuous isometric action of $G$ on a non-empty metric space $X$. Pick any $x \in X$. Since $G$ has small invariant neighborhoods, \begin{displaymath}
		\UCB(X) \, \longrightarrow \, \UCB(G,\mathscr{E}_{\Rsh}(G)), \quad f \, \longmapsto \, (g \mapsto f(gx))
\end{displaymath} constitutes a well-defined operator, which is moreover unital, positive, and $G$-equivariant with respect to the left-translation action on $G$ (for details, see~\cite[Lem.~3.6.5, p.~71]{PestovBook} or~\cite[Prop.~3.9]{JuschenkoSchneider}). Thus, via composition with this operator, any $G$-left-invariant mean on $\UCB(G,\mathscr{E}_{\Rsh}(G))$ gives rise to a $G$-invariant mean on $\UCB(X)$. \end{proof}

\section{Continuous rings and their unit groups}\label{section:continuous.rings}

We recollect some elements of von Neumann's continuous geometry~\cite{VonNeumannBook}. By a \emph{lattice} we mean a partially ordered set $L$ in which every pair of elements $x,y \in L$ admits both a (necessarily unique) supremum $x\vee y \in L$ and a (necessarily unique) infimum $x\wedge y \in L$. A \emph{complete lattice} is a partially ordered set $L$ such that every subset $S \subseteq L$ has a (necessarily unique) supremum $\bigvee S \in L$. If $L$ is a complete lattice, then every $S \subseteq L$ admits a (necessarily unique) infimum $\bigwedge S \in L$, too. A lattice $L$ is called \emph{bounded} if it has both a (necessarily unique) greatest element $1 = 1_{L} \in L$ and a (necessarily unique) least element $0 = 0_{L} \in L$. Clearly, any complete lattice is bounded. A lattice $L$ is said to be \emph{(directly) irreducible} if $\vert L \vert \geq 2$ and $L$ is not isomorphic to a direct product of two lattices of cardinality at least two. A \emph{continuous geometry} is a complete lattice $L$ such that \begin{itemize}
	\item[---] $L$ is \emph{complemented}, i.e., \begin{displaymath}
					\qquad \forall x \in L \ \exists y \in L \colon \quad x \vee y = 1, \ \, x\wedge y = 0 ,
				\end{displaymath}
	\item[---] $L$ is \emph{modular}, i.e., \begin{displaymath}
					\qquad \forall x,y,z \in L \colon \quad x \leq y \ \, \Longrightarrow \ \, x \vee (y \wedge z) = y \wedge (x \vee z) ,
				\end{displaymath}
	\item[---] and, for every chain $C \subseteq L$ and every element $x \in L$, \begin{displaymath}
		\qquad x \wedge \bigvee C \, = \, \bigvee \{ x \wedge y \mid y \in C\}, \quad \, x \vee \bigwedge C \, = \, \bigwedge \{ x \vee y \mid y \in C\} .
	\end{displaymath}
\end{itemize} A \emph{dimension function} on a bounded lattice $L$ is a map $\Delta \colon L \to [0,1]$ such that \begin{itemize}
	\item[---] $\Delta (0_{L}) = 0$ and $\Delta (1_{L}) = 1$,
	\item[---] $\Delta (x\vee y) + \Delta (x \wedge y) \, = \, \Delta (x) + \Delta (y)$ for all $x,y \in L$,
	\item[---] $\Delta$ is \emph{strictly monotone}, i.e., \begin{displaymath}
		\qquad \forall x,y \in L \colon \quad x < y \ \, \Longrightarrow \ \, \Delta (x) < \Delta (y).
\end{displaymath} \end{itemize} If $\Delta \colon L \to [0,1]$ is a dimension function on a bounded lattice $L$, then \begin{displaymath}
	\delta_{\Delta} \colon \, L \times L \, \longrightarrow \, [0,1] , \quad (x,y) \, \longmapsto \, \Delta(x \vee y)-\Delta(x\wedge y)
\end{displaymath} is a metric on $L$ (see~\cite[V.7, Lem.~on p.~76]{BirkhoffBook} or~\cite[I.6, Satz~6.2, p.~46]{MaedaBook}). By work of von Neumann~\cite{VonNeumannBook}\footnote{Existence is due to~\cite[I.VI, Thm.~6.9, p.~52]{VonNeumannBook} (see also~\cite[V.2, Satz~2.1, p.~118]{MaedaBook}), uniqueness is due to~\cite[I.VII, Cor.~1 on p.~60]{VonNeumannBook} (see also~\cite[V.2, Satz~2.3, p.~120]{MaedaBook}).}, every irreducible continuous geometry $L$~admits a unique dimension function, which will be denoted by $\Delta_{L} \colon L \to [0,1]$. If $L$ is an irreducible continuous geometry, then we let $\delta_{L} \defeq \delta_{\Delta_{L}}$.

We proceed to some basic remarks concerning von Neumann's continuous rings~\cite{VonNeumannBook} (see also~\cite{MaedaBook,GoodearlBook}). A ring will be called \emph{(directly) irreducible} if it is non-zero and not isomorphic to a direct product of two non-zero rings. Given a unital ring $R$, we consider the set \begin{displaymath}
	\lat(R) \, \defeq \, \{ aR \mid a \in R \} ,
\end{displaymath} partially ordered by set-theoretic inclusion. A unital ring $R$ is called \emph{(von Neumann) regular} if \begin{displaymath}
	\forall a \in R \ \exists b \in R \colon \qquad aba \, = \, a .
\end{displaymath} Due to~\cite[II.II, Thm.~2.4, p.~72]{VonNeumannBook}, if $R$ is a regular ring, then the partially ordered set $\lat(R)$ is a complemented, modular lattice, in which \begin{displaymath}
	I\vee J \, = \, I+J, \quad I\wedge J \, = \, I\cap J \qquad (I,J \in \lat(R)) .
\end{displaymath}

\begin{thm}[von Neumann~\cite{VonNeumannBook}]\label{theorem:irreducibility} A regular ring $R$ is irreducible if and only if $\lat(R)$ is irreducible. \end{thm}

\begin{proof} This is established in~\cite[II.II, Thm.~2.9, p.~76]{VonNeumannBook}. \end{proof}

A \emph{continuous} ring is a regular ring $R$ such that $\lat(R)$ is a continuous geometry. A \emph{rank function} on a regular $R$ is a map $\rk \colon R \to [0,1]$ such that \begin{itemize}
	\item[---] $\rk(1)=1$,
	\item[---] $\rk(ab) \leq \min \{ \rk(a),\rk(b)\}$ for all $a,b \in R$,
	\item[---] for all $e,f \in R$, \begin{displaymath}
				\qquad e^{2}=e, \ f^{2}=f, \ ef=fe=0 \ \, \Longrightarrow \ \, \rk(e+f) = \rk(e) + \rk(f),
			\end{displaymath}
	\item[---] $\rk (a) > 0$ for every $a \in R\setminus \{ 0 \}$.\footnote{The third condition readily entails that $\rk(0) = 0$.}
\end{itemize} For any rank function $\rk \colon R \to [0,1]$ on a regular ring $R$, \begin{displaymath}
	d_{\rk} \colon \, R \times R \, \longrightarrow \, [0,1], \quad (a,b) \, \longmapsto \, \rk(a-b) 
\end{displaymath} constitutes a metric on $R$ (see~\cite[II.XVIII, Lem.~18.1, pp. 231--232]{VonNeumannBook} or~\cite[VI.5, Satz~5.1, p.~154]{MaedaBook}).

\begin{thm}[von Neumann~\cite{VonNeumannBook}]\label{theorem:rank.function} If $R$ is an irreducible, continuous ring, then \begin{displaymath}
	\rk_{R} \colon \, R \, \longrightarrow \, [0,1], \quad a \, \longmapsto \, \Delta_{\lat(R)}(aR)
\end{displaymath} is the unique rank function on $R$. \end{thm}

\begin{proof} If $R$ is an irreducible, continuous ring, then $\rk_{R}$ constitutes a rank function on $R$ due to~\cite[II.XVII, Thm.~17.1, p.~224]{VonNeumannBook} and is unique as such by~\cite[II.XVII, Thm.~17.2, p.~226]{VonNeumannBook} \end{proof}

Let $R$ be an irreducible, continuous ring. Then $R$ is complete with respect to the \emph{rank metric} $d_{\rk, R} \defeq d_{\rk_{R}}$ according to~\cite[II.XVII, Thm.~17.4, p.~230]{VonNeumannBook}. The topology on $R$ generated by $d_{\rk, R}$ will be referred to as the \emph{rank topology} of $R$. The unit group \begin{displaymath}
	\GL(R) \, \defeq \, \{ a \in R \mid \exists b \in R \colon \, ab=ba=1 \}
\end{displaymath} endowed with the relative rank topology constitutes a topological group (cf.~\cite[Rem.~7.8]{FMS22}), which will be denoted by $\GL(R)_{\rk}$. Since its topology is generated by a bi-invariant metric (namely, the restriction of $d_{\rk,R}$), the topological group $\GL(R)_{\rk}$ has small invariant neighborhoods. 

\begin{lem}\label{lemma:isometric.action} Let $R$ be an irreducible, continuous ring. Then \begin{displaymath}
	\GL(R) \times \lat(R) \, \longrightarrow \, \lat(R), \quad (g,I) \, \longmapsto \, gI
\end{displaymath} is a continuous isometric action of $\GL(R)_{\rk}$ on $(\lat(R),\delta_{\lat(R)})$. Furthermore, for each $t \in [0,1]$, \begin{displaymath}
	\lat_{t}(R) \, \defeq \, \{ I \in \lat(R) \mid \Delta_{\lat(R)}(I) = t \} \, = \, \{ aR \mid a \in R, \, \rk_{R}(a) = t \} 
\end{displaymath} is $\GL(R)$-invariant.\end{lem}

\begin{proof} Evidently, $\GL(R) \times \lat(R) \to \lat(R), \, (g,I) \mapsto gI$ is a well-defined action. If $g \in \GL(R)$, then
\begin{equation}\tag{1}\label{unit}
	\Delta_{\lat(R)}(gaR) \, \stackrel{\ref{theorem:rank.function}}{=} \, \rk_{R}(ga) \, = \, \rk_{R}(a) \, \stackrel{\ref{theorem:rank.function}}{=} \, \Delta_{\lat(R)}(aR)
\end{equation} for all $a \in R$, thus \begin{align*}
	\delta_{\lat(R)}(gI,gJ) \, &= \, \Delta_{\lat(R)}(gI + gJ)-\Delta_{\lat(R)}(gI \cap gJ) \\
		& = \, \Delta_{\lat(R)}(g(I + J))-\Delta_{\lat(R)}(g(I \cap J)) \\
		& \stackrel{\eqref{unit}}{=} \, \Delta_{\lat(R)}(I + J)-\Delta_{\lat(R)}(I \cap J) \, = \, \delta_{\lat(R)}(I,J)
\end{align*} for all $I,J \in \lat(R)$, which shows that the considered action is isometric. Furthermore, as proved in~\cite[Lem.~7.9(3)]{FMS22}, \begin{equation}\tag{2}\label{orbital}
	\forall I \in \lat(R) \ \forall a,b \in R \colon \qquad \delta_{\lat(R)}(aI,bI) \, \leq \, 2\min \{ \rk_{R}(a-b), \Delta_{\lat(R)}(I)\} .
\end{equation} In turn, \begin{displaymath}
	\delta_{\lat(R)}(gI,hI) \, \stackrel{\eqref{orbital}}{\leq} \, 2\rk_{R}(g-h) \, = \, 2d_{\rk,R}(g,h)
\end{displaymath} for all $I \in \lat(R)$ and $g,h \in \GL(R)$. This means that, for each $I \in \lat(R)$, the map \begin{displaymath}
	(\GL(R),d_{\rk,R}) \, \longrightarrow \, (\lat(R),\delta_{\lat(R)}), \quad g \, \longmapsto \, gI
\end{displaymath} is $2$-Lipschitz, in particular continuous. Since the action is also isometric, thus the map $\GL(R) \times \lat(R) \to \lat(R), \, (g,I) \mapsto gI$ is continuous. The final assertion is an immediate consequence of Theorem~\ref{theorem:rank.function} and~\eqref{unit}. \end{proof}

An irreducible, continuous ring $R$ will be called \emph{discrete} if the rank topology of $R$ is discrete.

\begin{remark}[\cite{VonNeumannBook}, I.VII, Thm.~7.3, p.~58]\label{remark:discreteness} If $L$ is an irreducible continuous geometry, then either $\Delta_{L}(L) = [0,1]$, or there exists $n \in \N_{>0}$ with \begin{displaymath}
	\Delta_{L}(L) \! \left. \, = \, \left\{ \tfrac{k}{n} \, \right\vert k \in \{ 0,\ldots,n \} \right\} .
\end{displaymath} This readily implies that an irreducible, continuous ring $R$ is non-discrete if and only if $\rk_{R}(R) = [0,1]$. \end{remark}

\begin{lem}\label{lemma:skew.amenable} Let $R$ be a non-discrete irreducible, continuous ring, let $t \in [0,1]$. If $\GL(R)_{\rk}$ is amenable, then the action of $\GL(R)$ on $(\lat_{t}(R),\delta_{\lat(R)})$ is Eymard--Greenleaf amenable. \end{lem}

\begin{proof} Note that $\lat_{t}(R) \ne \emptyset$ thanks to Remark~\ref{remark:discreteness} and non-discreteness of~$R$. Since the topological group $\GL(R)_{\rk}$ has small invariant neighborhoods, thus the claim is a direct consequence of Lemma~\ref{lemma:isometric.action} and Lemma~\ref{lemma:pestov}. \end{proof}

For the sake of a transparent exposition, we conclude this section with a clarifying remark about discrete irreducible, continuous rings.

\begin{remark}[von Neumann~\cite{VonNeumannBook}]\label{remark:matrix.rings} A ring is a discrete irreducible, continuous ring if and only if it is isomorphic to a matrix ring $\Mat_{n}(D)$ for some division ring $D$ and some $n \in \N_{>0}$. We sketch the proof of this fact.
	
($\Longleftarrow$) Consider a division ring $D$ and let $n \in \N_{>0}$. Then $R \defeq \Mat_{n}(D)$ constitutes an irreducible, continuous ring due to~\cite[II.II, Thm~2.13, p.~81]{VonNeumannBook} and~\cite[IX.2, Satz~2.1, p.~185]{MaedaBook}. The uniqueness assertion of Theorem~\ref{theorem:rank.function} and the relevant properties of the normalization of the natural rank map on $R$ (see~\cite[I.10.12, p.~359--360]{bourbaki}) then imply that $\left. \rk_{R}(R) = \left\{ \tfrac{k}{n} \, \right\vert k \in \{ 0,\ldots,n \} \right\}$. In particular, the rank topology of $R$ is discrete.

($\Longrightarrow$) Let $R$ be a discrete irreducible, continuous ring. By Remark~\ref{remark:discreteness}, the set $\rk_{R}(R) = \Delta_{\lat(R)}(\lat(R))$ is finite. As $\Delta_{\lat(R)}$ is strictly monotone, it follows that every upward (resp., downward) directed subset of $\lat(R)$ has a greatest (resp., least) element. We deduce that every right ideal of $R$ belongs to $\lat(R)$: if $I$ is a right ideal of $R$, then we consider the upward directed set $\mathscr{J}$ of all finitely generated right ideals of $R$ contained in $I$, and we note that $\mathscr{J} \subseteq \lat(R)$ by regularity of $R$ (see~\cite[II.II, Thm.~2.3, p.~71]{VonNeumannBook}), which entails that $\mathscr{J}$ has a greatest element, whence $I = \bigcup \mathscr{J} \in \lat(R)$. Now, since $\lat(R)$ coincides with the set of all right ideals of $R$, our observation about downward directed subsets of $\lat(R)$ implies that $R$ is right Artinian. Furthermore, $R$ is simple by~\cite[VII.3, Hilfssatz~3.1, p.~166]{MaedaBook} (see also~\cite[Cor.~13.26, p.~170]{GoodearlBook}), thus the desired conclusion follows by the Artin--Wedderburn theorem. \end{remark}

\section{Geometry of projections and affiliated operators}\label{section:projections}

In this section we prove that, if the unit group of the ring of operators affiliated with a $\mathrm{II}_{1}$ factor $M$ is amenable with respect to the rank topology, then for any $t \in [0,1]$ the action of the unitary group of $M$ on the space of projections of $M$ of trace $t$ is Eymard--Greenleaf amenable (Lemma~\ref{lemma:amenable}).

We start off with some very general remarks on von Neumann algebras. For background, the reader is referred to~\cite{KadisonRingrose,ConwayBook}. Given a von Neumann algebra $M$, we consider its \emph{unitary group} \begin{displaymath}
	\U(M) \, \defeq \, \{ u \in M \mid uu^{\ast} = u^{\ast}u = 1 \} ,
\end{displaymath} as well as the set \begin{displaymath}
	\Pro(M) \, \defeq \, \left\{ p \in M \left\vert \, p^{2} = p = p^{\ast} \right\} \right.
\end{displaymath} of all \emph{projections} of $M$. If $M$ is a von Neumann factor of type $\mathrm{II}_{1}$, then we let $\tr_{M} \colon M \to \C$ denote its (necessarily faithful, normal) unique tracial state (cf.~\cite[Thm.~8.2.8, p.~517]{KadisonRingrose}), which in turn gives rise to the \emph{trace metric} \begin{displaymath}
	d_{\tr,M} \colon \, M \times M \, \longrightarrow \, \R_{\geq 0} , \quad (x,y) \, \longmapsto \, \sqrt{\tr_{M}((x-y)^{\ast}(x-y))} .
\end{displaymath}

\begin{remark}\label{remark:isometric.action} Let $M$ be a von Neumann factor of type $\mathrm{II}_{1}$. Then \begin{displaymath}
	\U(M) \times \Pro(M) \, \longrightarrow \, \Pro(M) , \quad (u,p) \, \longmapsto \, upu^{\ast}
\end{displaymath} is an isometric action of $\U(M)$ on $(\Pro(M),d_{\tr,M})$. For each $t \in [0,1]$, \begin{displaymath}
	\Pro_{t}(M) \, \defeq \, \{ p \in \Pro(M) \mid \tr_{M}(p) = t \}
\end{displaymath} is a $\U(M)$-invariant subset of $\Pro(M)$. Of course, $\Pro_{0}(M) = \{ 0 \}$ and $\Pro_{1}(M) = \{ 1 \}$. \end{remark}

The following remark summarizes several facts about the geometry of projections in $\mathrm{II}_{1}$ factors.

\begin{remark}\label{remark:projection.lattice} Let $M$ be a von Neumann algebra. We equip $\Pro(M)$ with the partial order defined by \begin{displaymath}
	p \leq q \quad :\Longleftrightarrow \quad qp = p \qquad (p,q \in \Pro(M)) .
\end{displaymath} Observe that, for any two $p,q \in \Pro(M)$, \begin{equation}\tag{$\dagger$}\label{swap}
	p \leq q \ \ \Longleftrightarrow \ \ qp=p \ \ \Longleftrightarrow \ \ (qp)^{\ast} = p^{\ast} \ \ \Longleftrightarrow \ \ p^{\ast}q^{\ast} = p^{\ast} \ \ \Longleftrightarrow \ \ pq = p .
\end{equation} Then $\Pro(M)$ is a complete lattice on which the map \begin{displaymath}
	\Pro(M) \, \longrightarrow \, \Pro(M), \quad p \, \longmapsto \, 1-p
\end{displaymath} constitutes an orthocomplementation (see~\cite[Prop.~6.3, p.~82]{redei}). Suppose now that $M$ is finite. Then $\Pro(M)$ is also modular (see~\cite[Prop.~6.14, p.~99]{redei}), thus a continuous geometry by~\cite{kaplansky}. Moreover, $M$ is a non-zero factor if and only if $\Pro(M)$ is irreducible (cf.~\cite[1.1, \S6, Ex.~11C, p.~39]{BerberianBook}), in which case \begin{displaymath}
	\Delta_{\Pro(M)} \, = \, {\tr_{M}}\vert_{\Pro(M)}
\end{displaymath} (see~\cite[8.4, p.~530]{KadisonRingrose}). In particular, if $M$ is a factor of type $\mathrm{II}_{1}$, then \begin{displaymath}
	\Delta_{\Pro(M)}(\Pro(M)) \, = \, \tr_{M}(\Pro(M)) \, = \, [0,1]
\end{displaymath} (see~\cite[Thm.~8.4.4(ii), p.~533]{KadisonRingrose}). \end{remark}

Now let $M$ be a finite von Neumann algebra acting on a Hilbert space~$H$. Then $\aff(M)$ is defined as the set of all densely defined, closed, linear operators on $H$ \emph{affiliated with $M$}, i.e., those commuting with every unitary in the commutant of $M$. That is, a densely defined, closed, linear operator $a$ on $H$ belongs to $\aff(M)$ if and only if $ua=au$ for every $u \in \U(M')$ (which entails that the domain of $a$ is $\U(M')$-invariant).

\begin{thm}[Murray \& von Neumann~\cite{MurrayVonNeumann}]\label{theorem:affiliated} Let $M$ be a finite von Neumann algebra. Then $\aff(M)$, equipped with the addition \begin{displaymath}
	\aff(M) \times \aff(M) \, \longrightarrow \, \aff(M), \quad (a,b) \, \longmapsto \, \overline{a+b}
\end{displaymath} and the multiplication \begin{displaymath}
	\aff(M) \times \aff(M) \, \longrightarrow \, \aff(M), \quad (a,b) \, \longmapsto \, \overline{ab} ,
\end{displaymath} is a unital ring, of which $M$ constitutes a unital subring. \end{thm}

\begin{proof} This is due to~\cite[Thm.~XV, p.~229]{MurrayVonNeumann} (see also~\cite[Thm.~6.13]{KadisonLiu}). \end{proof}

The reader is referred to~\cite[Sect.~6.2, pp.~32--36]{KadisonLiu} for a comprehensive account on and to~\cite{Berberian57,BerberianBook,Berberian82} for alternative algebraic descriptions of the rings constructed above. We confine ourselves to the following proposition, isolating the information relevant for our purposes.

\begin{prop}[von Neumann~\cite{VonNeumannBook}, Feldman~\cite{feldman}]\label{proposition:affiliated} Let $M$ be a finite von Neumann algebra. Then $\aff(M)$ is a continuous ring, and \begin{displaymath}
	\kappa_{M} \colon \, \Pro(M) \, \longrightarrow \, \lat(\aff(M)), \quad p \, \longmapsto \, p\aff(M)
\end{displaymath} is an order isomorphism. \end{prop}

\begin{proof} By~\cite[Thm.~2]{feldman} (which is based on~\cite[II.II, Appx.~2, (VI), p.~89--90]{VonNeumannBook}), the ring $\aff(M)$ is regular and the mapping $\kappa_{M}$ is an order isomorphism. Consequently, $\lat(\aff(M)) \cong \Pro(M)$ is a continuous geometry by Remark~\ref{remark:projection.lattice}, wherefore $\aff(M)$ is indeed a continuous ring. \end{proof}

\begin{remark}\label{remark:affiliated} Let $M$ be a von Neumann factor of type $\mathrm{II}_{1}$. Then \begin{itemize}
	\item[$(1)$] $\aff(M)$ is irreducible by Proposition~\ref{proposition:affiliated}, Remark~\ref{remark:projection.lattice}, Theorem~\ref{theorem:irreducibility},
	\item[$(2)$] ${\tr_{M}}\vert_{\Pro(M)} \stackrel{\ref{remark:projection.lattice}}{=} \Delta_{\Pro(M)} \stackrel{\ref{proposition:affiliated}}{=} {\Delta_{\lat(\aff(M))}} \circ {\kappa_{M}} \stackrel{\ref{theorem:rank.function}}{=} {\rk_{\aff(M)}}\vert_{\Pro(M)}$,
	\item[$(3)$] $\aff(M)$ is non-discrete, since \begin{displaymath}
					\qquad \rk_{\aff(M)}(\aff(M)) \, \stackrel{(2)}{=} \, \tr_{M}(\Pro (M)) \, \stackrel{\ref{remark:projection.lattice}}{=} \, [0,1] .
				\end{displaymath} 
\end{itemize} \end{remark}

The map from Proposition~\ref{proposition:affiliated} has the following additional properties.

\begin{lem}\label{lemma:isometric} Let $M$ be a $\mathrm{II}_{1}$ factor and let $R \defeq \aff(M)$. Then \begin{itemize}
	\item[$(1)$] $\kappa_{M} \colon \Pro(M) \to \lat(R)$ is $\U(M)$-equivariant,
	\item[$(2)$] $\kappa_{M}(\Pro_{t}(M)) = \lat_{t}(R)$ for each $t \in [0,1]$, and
	\item[$(3)$] $\kappa_{M}^{-1} \colon (\lat(R),\delta_{\lat(R)}) \to (\Pro(M),d_{\rk,R})$ is $1$-Lipschitz.
\end{itemize} \end{lem}

\begin{proof} (1) For all $p \in \Pro(M)$ and $u \in \U(M)$, \begin{displaymath}
	\kappa_{M}(upu^{\ast}) \, = \, upu^{\ast}R \, = \, upR \, = \, u\kappa_{M}(p) .
\end{displaymath} 

(2) This is a direct consequence of Proposition~\ref{proposition:affiliated} and Remark~\ref{remark:affiliated}(2).

(3) Let $I,J \in \lat(R)$. Consider $p \defeq \kappa_{M}^{-1}(I),\, q \defeq \kappa_{M}^{-1}(J) \in \Pro(M)$. Straightforward calculations using Remark~\ref{remark:projection.lattice}\eqref{swap} and the fact that $p\wedge q \leq p\vee q$ show that $e \defeq (p\vee q) - (p \wedge q) \in \Pro(M)$ and $(p\wedge q)e = e(p \wedge q) = 0$. 
Thus, \begin{equation}\tag{$\ast$}\label{difference}
	\rk_{R}(p \vee q) \, = \, \rk_{R}(e + (p\wedge q)) \, = \, \rk_{R}(e) + \rk_{R}(p\wedge q) .
\end{equation} Moreover, \begin{equation}\tag{$\ast\ast$}\label{null}
	(p-q)(p\wedge q) \, = \, p(p\wedge q) - q(p\wedge q) \, = \, (p\wedge q) - (p\wedge q) \, = \, 0 .
\end{equation} We conclude that \begin{align*}
	d_{\rk,R}\left(\kappa_{M}^{-1}(I),\kappa_{M}^{-1}(J)\right) \, &= \, d_{\rk,R}(p,q) \, = \, \rk_{R}(p-q) \, = \, \rk_{R}(p(p\vee q) - q(p\vee q)) \\
		& = \, \rk_{R}((p-q)(p\vee q)) \, = \, \rk_{R}((p-q)(e + (p\wedge q))) \\
		& = \, \rk_{R}((p-q)e + (p-q)(p\wedge q)) \, \stackrel{\eqref{null}}{=} \, \rk_{R}((p-q)e) \\
		& \leq \, \rk_{R}(e) \, \stackrel{\eqref{difference}}{=} \, \rk_{R}(p \vee q) - \rk_{R}(p\wedge q) \\
		& \stackrel{\ref{remark:affiliated}(2)}{=} \, \Delta_{\lat(R)}(\kappa_{M}(p \vee q)) - \Delta_{\lat(R)}(\kappa_{M}(p \wedge q))  \\
		& \stackrel{\ref{proposition:affiliated}}{=} \, \Delta_{\lat(R)}(\kappa_{M}(p) + \kappa_{M}(q)) - \Delta_{\lat(R)}(\kappa_{M}(p)\cap \kappa_{M}(q)) \\
		& = \, \Delta_{\lat(R)}(I + J) - \Delta_{\lat(R)}(I\cap J) \, = \, \delta_{\lat(R)}(I,J) . \qedhere
\end{align*} \end{proof}

Our proof of Lemma~\ref{lemma:amenable} also uses the following well-known inequality.

\begin{lem}\label{lemma:trace.topology} Let $M$ be a von Neumann factor of type $\mathrm{II}_{1}$. For every $a \in M$, \begin{displaymath}
	\tr_{M}(a^{\ast}a) \, \leq \, \Vert a \Vert^{2}\rk_{\aff(M)}(a) .
\end{displaymath} \end{lem}

\begin{proof} First, being a positive linear functional on a $C^{\ast}$-algebra, $\tr_{M}$ satisfies \begin{equation}\tag{$1$}\label{murphy}
	\forall x,y \in M \colon \qquad \tr_{M}(y^{\ast} x^{\ast}xy) \, \leq \, \Vert x^{\ast}x \Vert \tr_{M}(y^{\ast}y) 
\end{equation} (see, e.g.,~\cite[Thm.~3.3.7, p.~90]{MurphyBook}). Now, consider $R \defeq \aff(M)$ and let $a \in M$. By Proposition~\ref{proposition:affiliated}, there exists $p \in \Pro(M)$ with $pR = aR$. It follows that \begin{equation}\tag{$2$}\label{ideal}
	a \, = \, pa 
\end{equation} and \begin{equation}\tag{$3$}\label{dimension.vs.trace}
	\rk_{R}(a) \, \stackrel{\ref{theorem:rank.function}}{=} \, \Delta_{\lat(R)} (aR) \, = \, \Delta_{\lat(R)} (pR) \, = \Delta_{\lat(R)}(\kappa_{M}(p)) \, \stackrel{\ref{remark:affiliated}(2)}{=} \, \tr_{M}(p) .
\end{equation} We conclude that \begin{align*}
	\tr_{M}(a^{\ast}a) \, =\, \tr_{M}(aa^{\ast}) \, &\stackrel{\eqref{ideal}}{=} \, \tr_{M}(paa^{\ast}p^{\ast}) \\
	& \stackrel{\eqref{murphy}}{\leq} \, \Vert aa^{\ast} \Vert \tr_{M}(pp^{\ast}) \, = \, \Vert a \Vert^{2}\tr_{M}(p) \, \stackrel{\eqref{dimension.vs.trace}}{=} \, \Vert a \Vert^{2}\rk_{R}(a) . \qedhere
\end{align*} \end{proof}

\begin{lem}\label{lemma:amenable} Let $M$ be a $\mathrm{II}_{1}$ factor, let $R \defeq \aff(M)$, and let $t \in [0,1]$. Furthermore, consider the following conditions: \begin{itemize}
	\item[$(1)$] $\GL(R)_{\rk}$ is amenable.
	\item[$(2)$] $\GL(R) \curvearrowright (\lat_{t}(R),\delta_{\lat(R)})$ is Eymard--Greenleaf amenable
	\item[$(3)$] $\U(M) \curvearrowright (\Pro_{t}(M),d_{\rk,R})$ is Eymard--Greenleaf amenable.
	\item[$(4)$] $\U(M) \curvearrowright (\Pro_{t}(M),d_{\tr,M})$ is Eymard--Greenleaf amenable.
\end{itemize} Then, $(1)\Longrightarrow (2)\Longrightarrow (3)\Longrightarrow (4)$. \end{lem}

\begin{proof} (1)$\Longrightarrow$(2). Since $R$ is non-discrete by Remark~\ref{remark:affiliated}(3), amenability of the topological group $\GL(R)_{\rk}$ implies Eymard--Greenleaf amenability of the action of $\GL(R)$ on $(\lat_{t}(R),\delta_{\lat(R)})$ due to Lemma~\ref{lemma:skew.amenable}
	
(2)$\Longrightarrow$(3). Suppose that the action of $\GL(R)$ on $(\lat_{t}(R),\delta_{\lat(R)})$ is Eymard--Greenleaf amenable, i.e., there is a $\GL(R)$-invariant mean \begin{displaymath}
	\mu \colon \, \UCB(\lat_{t}(R),\delta_{\lat(R)}) \, \longrightarrow \, \R .
\end{displaymath} In particular, $\mu$ is $\U(M)$-invariant. By Lemma~\ref{lemma:isometric}, \begin{displaymath}
	\UCB(\Pro_{t}(M),d_{\rk,R}) \, \longrightarrow \, \UCB(\lat_{t}(R),\delta_{\lat(R)}), \quad f \, \longmapsto \, f \circ \kappa_{M}^{-1}
\end{displaymath} is a well-defined, $\U(M)$-equivariant, positive, unital, linear operator, thus \begin{displaymath}
	\UCB(\Pro_{t}(M),d_{\rk,R}) \, \longrightarrow \, \R, \quad f \, \longmapsto \, \mu\left( f \circ \kappa_{M}^{-1} \right)
\end{displaymath} constitutes a well-defined $\U(M)$-invariant mean. Hence, the action of $\U(M)$ on $(\Pro_{t}(M),d_{\rk,R})$ is Eymard--Greenleaf amenable.

(3)$\Longrightarrow$(4). Since \begin{align*}
	d_{\tr, M}(p,q) \, &= \, \sqrt{\tr_{M}((p-q)^{\ast}(p-q))} \, \stackrel{\ref{lemma:trace.topology}}{\leq} \, \Vert p-q \Vert \sqrt{\rk_{R}(p-q)} \\
	& \leq \, 2 \sqrt{\rk_{R}(p-q)} \, = \, 2 \sqrt{d_{\rk, R}(p,q)}
\end{align*} for all $p,q \in P(M)$, we see that $\UCB(\Pro_{t}(M),d_{\tr,M}) \subseteq \UCB(\Pro_{t}(M),d_{\rk,R})$. Thus, via restriction, any $\U(M)$-invariant mean on $\UCB(\Pro_{t}(M),d_{\rk,R})$ gives rise to a $\U(M)$-invariant mean on $\UCB(\Pro_{t}(M),d_{\tr,M})$.
\end{proof}

\section{Group von Neumann algebras and inner amenability}\label{section:groups}

In this section we prove that a non-trivial ICC group acting in an amenable fashion (in the sense of Eymard--Greenleaf) on the space of projections of trace $t$ of its group von Neumann algebra for some $t \in (0,1)$ must be inner amenable (Theorem~\ref{theorem:first}). Combining this with Lemma~\ref{lemma:amenable}, we deduce that, if the unit group of the ring affiliated with the group von Neumann algebra of a non-trivial ICC group $G$ is amenable with respect to the rank topology, then $G$ must be inner amenable (Theorem~\ref{theorem:second}). This way, we produce examples of non-discrete irreducible, continuous rings whose unit groups are non-amenable with respect to the rank topology (Corollary~\ref{corollary:main}).

Let us recall the definition of a group von Neumann algebra. For background on this construction, the reader is referred to~\cite[7, \S43 $+$ \S53]{ConwayBook}. Let $G$ be a group and consider the complex Hilbert space $\ell^{2}(G) = \ell^{2}(G,\C)$ densely spanned by the standard orthonormal basis $(b_{g})_{g \in G}$ defined by \begin{displaymath}
	b_{g}(x) \, \defeq \, \begin{cases}
								\, 1 & \text{if } x=g, \\ 
								\, 0 & \text{otherwise}
							\end{cases} \qquad (g,x \in G).
\end{displaymath} As usual, the \emph{left regular representation} $\lambda_{G} \colon G \to \U(\B(\ell^{2}(G)))$ and the \emph{right regular representation} $\rho_{G} \colon G \to \U(\B(\ell^{2}(G)))$ are given by \begin{displaymath}
	\lambda_{G}(g)(f)(h) \, \defeq \, f(g^{-1}h), \qquad \rho_{G}(g)(f)(h) \, \defeq \, f(hg) \qquad \left(g,h \in G, \, f \in \ell^{2}(G)\right) ,
\end{displaymath} and the \emph{adjoint representation}~\cite{effros} is defined as \begin{displaymath}
	\alpha_{G} \colon \, G \, \longrightarrow \, \U(\B(\ell^{2}(G))), \quad g \, \longmapsto \, \lambda_{G}(g)\rho_{G}(g) = \rho_{G}(g)\lambda_{G}(g) .
\end{displaymath} The \emph{group von Neumann algebra} of $G$ is defined as the bicommutant \begin{displaymath}
	\vN(G) \, \defeq \, \lambda_{G}(G)'' \, \subseteq \, \B\left(\ell^{2}(G)\right)
\end{displaymath} and comes along equipped with the faithful, normal, tracial state \begin{displaymath}
	\tr_{\vN(G)} \colon \, \vN(G) \, \longrightarrow \, \C, \quad a \, \longmapsto \, \langle a(b_{e}), b_{e} \rangle .
\end{displaymath} Furthermore, let us consider the $\alpha_{G}$-invariant closed linear subspace \begin{displaymath}
	\left. \ell^{2}_{0}(G) \, \defeq \, \{ b_{e} \}^{\perp} \, = \, \left\{ f \in \ell^{2}(G) \, \right\vert \langle f,b_{e} \rangle = 0 \right\} \, = \, \left. \left\{ f \in \ell^{2}(G) \, \right\vert f(e) = 0 \right\} .
\end{displaymath} 

\begin{remark}\label{remark:icc} (1) A group $G$ is said to have the \emph{infinite conjugacy class property}, or to be an \emph{ICC group}, if the conjugacy class of every element of $G\setminus \{ e \}$ is infinite. It is well known (see, e.g.,~\cite[Thm.~43.13, p.~249]{ConwayBook}) that a group $G$ has the infinite conjugacy class property if and only if $\vN(G)$ is a factor. Moreover, if $G$ is a non-trivial ICC group, then the factor $\vN(G)$ is of type $\mathrm{II}_{1}$ (cf.~\cite[Thm.~53.1, p.~301]{ConwayBook}).

(2) Let $G$ be a non-trivial ICC group. Since $G \to \U(\vN(G)), \, g \mapsto \lambda_{G}(g)$ is a homomorphism, Remark~\ref{remark:isometric.action} entails that \begin{displaymath}
	G \times \Pro(\vN(G)) \, \longrightarrow \, \Pro(\vN(G)) , \quad (g,p) \, \longmapsto \, \lambda_{G}(g)p\lambda_{G}(g)^{\ast}
\end{displaymath} is an isometric action of $G$ on $(\Pro(\vN(G)),d_{\tr,\vN(G)})$, which leaves each of the sets $\Pro_{t}(\vN(G))$ ($t \in [0,1]$) invariant. Henceforth, this action will be referred to as the \emph{natural} action of $G$ on $\Pro(\vN(G))$ (or $\Pro_{t}(\vN(G))$, for $t \in [0,1]$, resp.). \end{remark}

For a Hilbert space $H$, we consider its unit sphere $\Sph_{H} \defeq \{ x \in H \mid \Vert x \Vert = 1 \}$ equipped with the induced metric $\Sph_{H} \times \Sph_{H} \to \R, \, (x,y) \mapsto \Vert x-y \Vert$.

\begin{lem}\label{lemma:map} Let $G$ be a non-trivial ICC group and let $t \in (0,1)$. Then the map \begin{displaymath}
	\phi_{G,t} \colon \, \Pro_{t}(\vN(G)) \, \longrightarrow \, \Sph_{\ell^{2}_{0}(G)}, \quad p \, \longmapsto \, \tfrac{1}{\sqrt{t-t^{2}}}(p(b_{e})-t b_{e})
\end{displaymath} is $\tfrac{1}{\sqrt{t-t^{2}}}$-Lipschitz with respect to the trace metric on $\Pro_{t}(\vN(G))$. Furthermore, \begin{displaymath}
	\phi_{G,t}(\lambda_{G}(g)p\lambda_{G}(g)^{\ast}) \, = \, \alpha_{G}(g)(\phi_{G,t}(p))
\end{displaymath} for all $p \in \Pro_{t}(\vN(G))$ and $g \in G$. \end{lem}

\begin{proof} First of all, we need to show that $\phi_{G,t}$ is well defined. To this end, let $p \in \Pro_{t}(\vN(G))$. Evidently, $t-t^{2} > 0$ as $t \in (0,1)$. Moreover, \begin{displaymath}
	\langle \phi_{G,t}(p),b_{e} \rangle \, = \, \tfrac{1}{\sqrt{t-t^{2}}}\left(\langle p(b_{e}),b_{e} \rangle - t\langle b_{e},b_{e} \rangle \right) \, = \, \tfrac{1}{\sqrt{t-t^{2}}}(\tr_{\vN(G)}(p) - t) \, = \, 0 ,
\end{displaymath} hence $\phi_{G,t}(p) \in \ell^{2}_{0}(G)$. Since $p \in \Pro(\vN(G))$, we infer that \begin{align*}
	\Vert \phi_{G,t}(p) \Vert_{2} \, & = \, \tfrac{1}{\sqrt{t-t^{2}}} \sqrt{\langle p(b_{e}), p(b_{e})\rangle + \langle tb_{e},tb_{e} \rangle - 2 \Re \langle p(b_{e}),tb_{e} \rangle} \\
	& = \, \tfrac{1}{\sqrt{t-t^{2}}} \sqrt{\langle p(b_{e}), b_{e}\rangle + t^{2}\langle b_{e},b_{e} \rangle - 2 t \Re \langle p(b_{e}),b_{e} \rangle} \\
		& = \, \tfrac{1}{\sqrt{t-t^{2}}} \sqrt{\tr_{\vN(G)}(p) + t^{2} - 2 t \Re \tr_{\vN(G)}(p)} \, = \, \tfrac{\sqrt{t - t^{2}}}{\sqrt{t-t^{2}}} \, = \, 1 ,
\end{align*} i.e., $\phi_{G,t}(p) \in \Sph_{\ell^{2}_{0}(G)}$. This shows that the map $\phi_{G,t}$ is well defined. Concerning Lipschitz continuity, we observe that \begin{align*}
	\Vert \phi_{G,t}(p) - \phi_{G,t}(q) \Vert_{2} \, & = \, \tfrac{1}{\sqrt{t-t^{2}}} \Vert (p-q)(b_{e}) \Vert_{2} \, = \, \tfrac{1}{\sqrt{t-t^{2}}} \sqrt{\langle (p-q)(b_{e}),(p-q)(b_{e}) \rangle} \\
	& = \, \tfrac{1}{\sqrt{t-t^{2}}} \sqrt{\langle (p-q)^{\ast}(p-q)(b_{e}),b_{e} \rangle} \\
	& = \, \tfrac{1}{\sqrt{t-t^{2}}} \sqrt{\tr_{\vN(G)}((p-q)^{\ast}(p-q))} \, = \, \tfrac{1}{\sqrt{t-t^{2}}}d_{\tr,\vN(G)}(p,q)
\end{align*} for all $p,q \in \Pro_{t}(\vN(G))$. Finally, since $\rho_{G}(G) \subseteq \lambda_{G}(G)' = \lambda_{G}(G)''' = \vN(G)'$, we see that, for all $p \in \Pro_{t}(\vN(G))$ and $g \in G$, \begin{align*}
	\phi_{G,t}(\lambda_{G}(g)p\lambda_{G}(g)^{\ast}) \, &= \, \tfrac{1}{\sqrt{t-t^{2}}}\left((\lambda_{G}(g)p\lambda_{G}(g)^{\ast})(b_{e})- tb_{e}\right) \\
	& = \, \tfrac{1}{\sqrt{t-t^{2}}}\left((\lambda_{G}(g)p)\left(b_{g^{-1}}\right)- tb_{e}\right) \\
	& = \, \tfrac{1}{\sqrt{t-t^{2}}}\left((\lambda_{G}(g)p\rho_{G}(g))(b_{e})- tb_{e}\right) \\
	& \stackrel{\rho_{G}(g) \in \vN(G)'}{=} \, \tfrac{1}{\sqrt{t-t^{2}}}\left((\lambda_{G}(g)\rho_{G}(g)p)(b_{e})- tb_{e}\right) \\
	& = \, \tfrac{1}{\sqrt{t-t^{2}}}\left((\alpha_{G}(g)p)(b_{e})- tb_{e}\right) \\
	& = \, \tfrac{1}{\sqrt{t-t^{2}}}\left((\alpha_{G}(g)p)(b_{e})- t\alpha_{G}(g)(b_{e})\right) \\
	& = \, \alpha_{G}(g)\left( \tfrac{1}{\sqrt{t-t^{2}}}(p(b_{e})- t b_{e}) \right) \, = \, \alpha_{G}(g)(\phi_{G,t}(p)) . \qedhere
\end{align*} \end{proof}

Before elaborating on consequences of Lemma~\ref{lemma:map}, let us recall another basic fact (Lemma~\ref{lemma:map.2}). To clarify some relevant notation, let $X$ be a set. Then $\Sym(X)$ denotes the full symmetric group over $X$, which consists of all bijections from $X$ to itself. Furthermore, let us equip the set \begin{displaymath}
	\left. \Prob(X) \, \defeq \, \left\{ f \in \ell^{1}(X,\R) \, \right\vert \Vert f \Vert_{1} = 1, \, f \geq 0 \right\} 
\end{displaymath} with the metric \begin{displaymath}
	\Prob(X) \times \Prob(X) \, \longrightarrow \, \R, \quad (f,g) \, \longmapsto \, \Vert f-g \Vert_{1} .
\end{displaymath}

\begin{remark}\label{remark:isometry} Let $G$ be a group and let $X \defeq G\setminus \{ e \}$. Consider the homomorphism $\gamma_{G} \colon G \to \Sym(X)$ given by \begin{displaymath}
	\gamma_{G}(g)(x) \, \defeq \, gxg^{-1} \qquad (g \in G, \, x \in X) .
\end{displaymath} Note that $\ell^{2}_{0}(G) \to \ell^{2}(X), \, f \mapsto f\vert_{X}$ is an isometric linear isomorphism, and \begin{displaymath}
	\forall g \in G \ \forall f \in \ell^{2}_{0}(G) \colon \quad \alpha_{G}(g)(f)\vert_{X} \, = \, (f\vert_{X}) \circ \gamma_{G}\left(g^{-1}\right) .
\end{displaymath} \end{remark}

The following lemma is essentially known in the theory of Banach spaces: the map discussed in Lemma~\ref{lemma:map.2} is a close relative of the \emph{Mazur map}, which serves as a uniform isomorphism between the unit spheres of the $\ell^{p}$-spaces for $1 \leq p < \infty$ (see~\cite[9.1, pp.~197--199]{BenyaminiLindenstrauss} for details). We include the short argument for the sake of convenience.

\begin{lem}\label{lemma:map.2} Let $X$ be a set. Then \begin{displaymath}
	\psi_{X} \colon \, \Sph_{\ell^{2}(X)} \, \longrightarrow \, \Prob (X), \quad f \, \longmapsto \, \vert f \vert^{2}
\end{displaymath} is $2$-Lipschitz. Also, $\psi_{X}(f \circ \sigma) = \psi_{X}(f) \circ \sigma$ for all $f \in \Sph_{\ell^{2}(X)}$ and $\sigma \in \Sym(X)$. \end{lem}

\begin{proof} First of all, let us note that $\psi_{X}$ is well defined: indeed, if $f \in \Sph_{\ell^{2}(X)}$, then $\vert f \vert^{2}(x) = \vert f(x) \vert^{2} \in \R_{\geq 0}$ for every $x \in X$ and also $\sum_{x \in X} \vert f\vert^{2}(x) = \Vert f \Vert_{2}^{2} = 1$, wherefore $\vert f \vert^{2} \in \Prob(X)$. Furthermore, by the Cauchy--Schwarz inequality, \begin{align*}
	\Vert \psi_{X}(f) - \psi_{X}(g) \Vert_{1} \, & = \, \left\lVert \vert f \vert^{2} - \vert g \vert^{2} \right\rVert_{1} \, = \, \sum\nolimits_{x \in X} \left\lvert \vert f(x)\vert^{2} - \vert g(x) \vert^{2} \right\rvert \\
	& = \, \sum\nolimits_{x \in X} \vert \, \vert f(x) \vert + \vert g(x)\vert \, \vert \cdot \vert \, \vert f(x)\vert - \vert g(x)\vert \, \vert \\
	& \leq \, \sum\nolimits_{x \in X} ( \vert f(x) \vert + \vert g(x)\vert) \vert f(x) - g(x) \vert \\
	& = \, \lvert \langle \vert f\vert + \vert g \vert ,\vert f - g \vert \rangle \rvert \, \leq \, \Vert \, \vert f \vert + \vert g \vert \, \Vert_{2}\cdot \Vert f - g \Vert_{2} \, \leq \, 2 \Vert f-g \Vert_{2}
\end{align*} for all $f,g \in \Sph_{\ell^{2}(X)}$. Finally, if $f \in \Sph_{\ell^{2}(X)}$ and $\sigma \in \Sym(X)$, then \begin{displaymath}
	\psi_{X}(f \circ \sigma) \, = \, \vert f \circ \sigma \vert^{2} \, = \, \vert f \vert^{2} \circ \sigma \, = \, \psi_{X}(f) \circ \sigma . \qedhere
\end{displaymath} \end{proof}

Now we turn back to the map devised in Lemma~\ref{lemma:map}.

\begin{lem}\label{lemma:map.3} Let $G$ be a non-trivial ICC group, let $X \defeq G\setminus \{ e\}$ and let $t \in (0,1)$. Then \begin{displaymath}
	\xi_{G,t} \colon \, \Pro_{t}(\vN(G)) \, \longrightarrow \, \Prob (X), \quad p \, \longmapsto \, \psi_{X}(\phi_{G,t}(p)\vert_{X})
\end{displaymath} is $\tfrac{2}{\sqrt{t-t^{2}}}$-Lipschitz with respect to the trace metric on $\Pro_{t}(\vN(G))$. Furthermore, \begin{displaymath}
	\xi_{G,t}(\lambda_{G}(g)p\lambda_{G}(g)^{\ast}) \, = \, \xi_{G,t}(p) \circ \gamma_{G}\left(g^{-1}\right)
\end{displaymath} for all $p \in \Pro_{t}(\vN(G))$ and $g \in G$. \end{lem}

\begin{proof} Thanks to Lemma~\ref{lemma:map}, Remark~\ref{remark:isometry} and Lemma~\ref{lemma:map.2}, the map $\xi_{G,t}$ is well defined. We also see that \begin{align*}
	\Vert \xi_{G,t}(p) - \xi_{G,t}(q) \Vert_{1} \, &= \, \Vert \psi_{X}(\phi_{G,t}(p)\vert_{X}) - \psi_{X}(\phi_{G,t}(q)\vert_{X}) \Vert_{1} \\
	& \stackrel{\ref{lemma:map.2}}{\leq} \, 2\Vert \phi_{G,t}(p)\vert_{X} - \phi_{G,t}(q)\vert_{X} \Vert_{2} \, \stackrel{\ref{remark:isometry}}{=} \, 2\Vert \phi_{G,t}(p) - \phi_{G,t}(q) \Vert_{2} \\
	& \stackrel{\ref{lemma:map}}{\leq} \, \tfrac{2}{\sqrt{t-t^{2}}}d_{\tr,\vN(G)}(p,q)
\end{align*} for all $p,q \in \Pro_{t}(\vN(G))$, that is, $\xi_{G,t}$ is $\tfrac{2}{\sqrt{t-t^{2}}}$-Lipschitz with respect to $d_{\tr,\vN(G)}$. Moreover, for all $g \in G$ and $p \in \Pro_{t}(\vN(G))$, \begin{align*}
	\xi_{G,t}(\lambda_{G}(g)p&\lambda_{G}(g)^{\ast}) \, = \, \psi_{X}(\phi_{G,t}(\lambda_{G}(g)p\lambda_{G}(g)^{\ast})\vert_{X}) \, \stackrel{\ref{lemma:map}}{=} \, \psi_{X}(\alpha_{G}(g)(\phi_{G,t}(p))\vert_{X}) \\
	& \stackrel{\ref{remark:isometry}}{=} \, \psi_{X}\left((\phi_{G,t}(p)\vert_{X}) \circ \gamma_{G}\left(g^{-1}\right)\right) \, \stackrel{\ref{lemma:map.2}}{=} \, \psi_{X}(\phi_{G,t}(p)\vert_{X}) \circ \gamma_{G}\left(g^{-1}\right) \\
	&= \, \xi_{G,t}(p) \circ \gamma_{G}\left(g^{-1}\right) .\qedhere
\end{align*} \end{proof}

\begin{lem}\label{lemma:map.4} Let $G$ be a non-trivial ICC group, let $X \defeq G\setminus \{ e\}$ and let $t \in (0,1)$. Then \begin{displaymath}
	\Xi_{G,t} \colon \, \ell^{\infty}(X,\R) \, \longrightarrow \, \UCB(\Pro_{t}(\vN(G)),d_{\tr,\vN(G)})
\end{displaymath} given by \begin{displaymath}
	\Xi_{G,t}(f)(p) \, \defeq \, \sum\nolimits_{x \in X}f(x)\xi_{G,t}(p)(x) \qquad \left(f \in \ell^{\infty}(X,\R), \, p \in \Pro_{t}(\vN(G))\right)
\end{displaymath} is a positive, unital, linear operator. Furthermore, \begin{displaymath}
	\Xi_{G,t}(f \circ {\gamma_{G}(g)})(p) \, = \, \Xi_{G,t}(f)(\lambda_{G}(g)p\lambda_{G}(g)^{\ast})
\end{displaymath} for all $f \in \ell^{\infty}(X,\R)$, $g \in G$ and $p \in \Pro_{t}(\vN(G))$. \end{lem}

\begin{proof} In order to prove that $\Xi_{G,t}$ is well defined, consider any $f \in \ell^{\infty}(X,\R)$. Since $\xi_{G,t}(\Pro_{t}(\vN(G))) \subseteq \Prob(X)$ by Lemma~\ref{lemma:map.3}, it follows that \begin{displaymath}
	\sup\nolimits_{p \in \Pro_{t}(\vN(G))} \vert \Xi_{G,t}(f)(p) \vert \, \leq \, \sup\nolimits_{p \in \Pro_{t}(\vN(G))} \sum\nolimits_{x \in X} \vert f(x)\vert \xi_{G,t}(p)(x) \, \leq \, \Vert f \Vert_{\infty},
\end{displaymath} thus $\Xi_{G,t}(f) \in \ell^{\infty}(\Pro_{t}(\vN(G)),\R)$. For all $p,q \in \Pro_{t}(\vN(G))$, we see that \begin{align*}
	\vert \Xi_{G,t}(f)(p) - \Xi_{G,t}(f)(q) \vert \, &\leq \, \sum\nolimits_{x \in X} \vert f(x) \vert \cdot \vert \xi_{G,t}(p)(x) - \xi_{G,t}(q)(x) \vert \\
	& \leq \, \Vert f \Vert_{\infty} \Vert \xi_{G,t}(p) - \xi_{G,t}(q) \Vert_{1} \, \stackrel{\ref{lemma:map.3}}{\leq} \, \, \tfrac{2}{\sqrt{t-t^{2}}}\Vert f \Vert_{\infty} d_{\tr,\vN(G)}(p,q) .
\end{align*} Thus, $\Xi_{G,t}(f) \colon \Pro_{t}(\vN(G)) \to \R$ is $\tfrac{2}{\sqrt{t-t^{2}}}\Vert f \Vert_{\infty}$-Lipschitz with respect to~$d_{\tr,\vN(G)}$. In particular, $\Xi_{G,t}(f) \in \UCB(\Pro_{t}(\vN(G)),d_{\tr,\vN(G)})$. Hence, $\Xi_{G,t}$ is well defined. It is straightforward to check that $\Xi_{G,t}$ is linear. As $\xi_{G,t}(\Pro_{t}(\vN(G))) \subseteq \Prob(X)$ again by Lemma~\ref{lemma:map.3}, the operator $\Xi_{G,t}$ is moreover unital and positive. Finally, for all $f \in \ell^{\infty}(X,\R)$, $g \in G$ and $p \in \Pro_{t}(\vN(G))$, \begin{align*}
	\Xi_{G,t}(f \circ {\gamma_{G}(g)})(p) \, &= \, \sum\nolimits_{x \in X} f(\gamma_{G}(g)(x))\xi_{G,t}(p)(x) \\
	& = \, \sum\nolimits_{x \in X} f(x)\xi_{G,t}(p)\left(\gamma_{G}\left(g^{-1}\right)(x)\right) \\
	& \stackrel{\ref{lemma:map.3}}{=} \, \sum\nolimits_{x \in X} f(x)\xi_{G,t}(\lambda_{G}(g)p\lambda_{G}(g)^{\ast})(x) \\
	& = \, \Xi_{G,t}(f)(\lambda_{G}(g)p\lambda_{G}(g)^{\ast}) . \qedhere
\end{align*} \end{proof}

Recall that a group $G$ is said to be \emph{inner amenable}~\cite{effros} if either $\vert G \vert = 1$ or the action of $G$ on the (discrete) set $G\setminus \{ e\}$ given by conjugation is amenable, i.e., there exists a $\gamma_{G}(G)$-invariant mean on $\ell^{\infty}(G\setminus \{ e\},\R)$. Note that every non-inner amenable group is a non-trivial ICC group.

\begin{thm}\label{theorem:first} Let $G$ be a non-trivial ICC group and let $t \in (0,1)$. If the natural action of $G$ on $(\Pro_{t}(\vN(G)),d_{\tr,\vN(G)})$ is Eymard--Greenleaf amenable, then $G$ is inner amenable. \end{thm}

\begin{proof} In the light of Remark~\ref{remark:icc}(2), for each $g \in G$, we consider \begin{displaymath}
	\pi_{G}(g) \colon \, \Pro_{t}(\vN(G)) \, \longrightarrow \, \Pro_{t}(\vN(G)), \quad p \, \longmapsto \, \lambda_{G}(g)p\lambda_{G}(g)^{\ast} .
\end{displaymath} Suppose that the action of $G$ on $(\Pro_{t}(\vN(G)),d_{\tr,\vN(G)})$ is Eymard--Greenleaf amenable, i.e., there is a mean $\mu \colon \UCB(\Pro_{t}(\vN(G)),d_{\tr,\vN(G)}) \to \R$ such that \begin{equation}\tag{$\ast$}\label{invariant.mean}
	\forall g \in G \ \forall f \in \UCB(\Pro_{t}(\vN(G)),d_{\tr,\vN(G)}) \colon \quad \mu(f \circ {\pi_{G}(g)}) \, = \, \mu(f) .
\end{equation} Since $\Xi_{G,t}$ is a a positive, unital, linear operator by Lemma~\ref{lemma:map.4}, \begin{displaymath}
	\nu \defeq \mu \circ {\Xi_{G,t}} \colon \, \ell^{\infty}(G\setminus \{ e\},\R) \, \longrightarrow \, \R
\end{displaymath} constitutes a mean. Furthermore, for all $g \in G$ and $f \in \ell^{\infty}(G\setminus \{ e \},\R)$, \begin{displaymath}
	\nu(f \circ {\gamma_{G}(g)}) \, = \, \mu(\Xi_{G,t}(f \circ {\gamma_{G}(g)})) \, \stackrel{\ref{lemma:map.4}}{=} \, \mu(\Xi_{G,t}(f) \circ {\pi_{G}(g)}) \, \stackrel{\eqref{invariant.mean}}{=} \, \mu(\Xi_{G,t}(f)) \, = \, \nu(f) .
\end{displaymath} Thus, $G$ is inner amenable. \end{proof}

\begin{thm}\label{theorem:second} Let $G$ be a non-trivial ICC group and let $R \defeq \aff(\vN(G))$. If the topological group $\GL(R)_{\rk}$ is amenable, then $G$ is inner amenable. \end{thm}

\begin{proof} Suppose that $\GL(R)_{\rk}$ is amenable. Then, by Lemma~\ref{lemma:amenable}, the action of $\U(\vN(G))$ on $(\Pro_{1/2}(\vN(G)),d_{\tr,\vN(G)})$ is Eymard--Greenleaf amenable, whence the natural action of $G$ on $(\Pro_{1/2}(\vN(G)),d_{\tr,\vN(G)})$ is Eymard--Greenleaf amenable, too. Hence, $G$ is inner amenable by Theorem~\ref{theorem:first}. \end{proof}

\begin{cor}\label{corollary:main} Let $G$ be a group that is not inner amenable. Then $\aff(\vN(G))$ is a non-discrete irreducible, continuous ring whose unit group is non-amenable with respect to the rank topology. \end{cor}

\begin{proof} Not being inner amenable, $G$ must be a non-trivial ICC group. Thus, Remark~\ref{remark:icc}(1), Proposition~\ref{proposition:affiliated} and Remark~\ref{remark:affiliated}(1)$+$(3) together assert that $R \defeq \aff(\vN(G))$ is a non-discrete irreducible, continuous ring. According to Theorem~\ref{theorem:second}, the topological group $\GL(R)_{\rk}$ is non-amenable. \end{proof}

For the sake of completeness, we mention two prominent results negating inner amenability for certain concrete groups, thus providing specific examples of continuous rings such as in Corollary~\ref{corollary:main}.

\begin{prop}[Effros~\cite{effros}]\label{proposition:robin} Let $X$ be a set with $\vert X \vert > 1$. Then the free group $\free (X)$ is not inner amenable. \end{prop}

\begin{proof}\!\!\!\footnote{This argument, which is simpler than the original one from~\cite{effros} (based on~\cite{MurrayVonNeumannIV}), was kindly pointed out to the author by Robin Tucker-Drob.} For each $g \in \free (X)\setminus \{ e \}$, the centralizer $\Cent_{\free (X)}(g) = \{ h \in \free (X) \mid gh=hg\}$ is cyclic\footnote{By the Nielsen--Schreier theorem, the subgroup $\Cent_{\free (X)}(g) \leq \free (X)$ is free, i.e., $\Cent_{\free (X)}(g) \cong \free (Y)$ for some set $Y$. Since the center of $\Cent_{\free (X)}(g)$ contains the non-trivial element $g$, we conclude that $\vert Y \vert = 1$. Hence, $\Cent_{\free (X)}(g) \cong \free (Y) \cong \Z$ is cyclic.}, thus amenable. Since the (discrete) group $\free (X)$ is non-amenable, this implies by~\cite[Cor.~4.3]{HaagerupOlesen} (which is a consequence of a result due to Rosenblatt~\cite[Prop.~3.5]{Rosenblatt81}) that $\free (X)$ is not inner amenable. \end{proof}

The proof of the following result in~\cite{HaagerupOlesen} has the same global structure as the one above, but requires a much more delicate analysis of centralizers.

\begin{thm}[Haagerup \& Olesen~\cite{HaagerupOlesen}]\label{theorem:haagerup.olesen} The Thompson groups $T$ and $V$ are not inner amenable. \end{thm}

\begin{remark} While the work of Carderi and Thom~\cite{CarderiThom} provides examples of non-discrete irreducible, continuous rings $R$ such that $\GL(R)_{\rk}$ is extremely amenable, our Theorem~\ref{theorem:second} exhibits instances of non-discrete irreducible, continuous rings $R$ such that $\GL(R)_{\rk}$ is non-amenable. In view of the different constructions of continuous rings employed in~\cite{CarderiThom} and the present note, it would be interesting to know \begin{itemize}
	\item[$(1)$] whether $\GL(\Mat_{\infty}(\Q))_{\rk}$ is amenable, and
	\item[$(2)$] whether $\GL(\aff(M))_{\rk}$ is amenable for some $\mathrm{II}_{1}$ factor $M$.
\end{itemize} \end{remark}

\section*{Acknowledgments}

The author is grateful to Andreas Thom, Maxime Gheysens, and Robin Tucker-Drob for their comments on earlier versions of this note. Moreover, the author would like to thank Robin Tucker-Drob for having pointed out the simple proof of Proposition~\ref{proposition:robin}. Finally, the author would like to express his sincere gratitude towards the anonymous referee for a number of insightful remarks and inspiring questions.


\end{document}